\documentclass[reqno]{amsart}

\usepackage{amssymb}

\usepackage[cmtip,all]{xy}
\SilentMatrices 

\copyrightinfo{2013}{J. D. Christensen and G. Wang}

\newtheorem{thm}{Theorem}[section]
\newtheorem{lem}[thm]{Lemma}
\newtheorem{pro}[thm]{Proposition}
\newtheorem{cor}[thm]{Corollary}
\newtheorem{conj}[thm]{Conjecture}

\theoremstyle{definition}
\newtheorem{de}[thm]{Definition}
\newtheorem{ex}[thm]{Example}

\theoremstyle{remark}
\newtheorem{rmk}[thm]{Remark}

\numberwithin{equation}{section}

\makeatletter
\def\subsection{\@startsection{subsection}{2}%
  \z@{.5\linespacing\@plus.7\linespacing}{.5\linespacing}%
  {\normalfont\bfseries}}
\makeatother

\let\leq\leqslant
\let\geq\geqslant

\usepackage[letterpaper,margin=1.25in]{geometry}

\newcommand{\dosomething}[1]{\textbf{[#1]}}
\newcommand{\donothing}[1]{}
\newcommand{\Comment}{\donothing}
\newcommand{\Commentson}{\renewcommand{\Comment}{\dosomething}}

\Commentson

\usepackage{verbatim}

\usepackage{enumerate}

\usepackage{tikz}
\tikzstyle{pre}=[<-,shorten <=2pt,shorten >=1.3pt,>=stealth,semithick]
\tikzstyle{post}=[->,shorten >=2pt,shorten <=1.3pt,>=stealth,semithick]
\tikzstyle{dot}=[circle, draw, fill=black!50, inner sep=0pt, minimum width=3pt]
\tikzstyle{X}=[above left=-2.5pt,font=\footnotesize]
\tikzstyle{Y}=[above right=-2.5pt,xshift=-1.8pt,font=\footnotesize]

\newcommand{\Dend}{;}

\usetikzlibrary{matrix}

\newcommand{\st}{\mid}

\newcommand{\stmod}[1]{\mathsf{stmod}(#1)}
\newcommand{\StMod}[1]{\mathsf{StMod}(#1)}
\newcommand{\modu}[1]{\mathsf{mod}(#1)}
\newcommand{\Mod}[1]{\mathsf{Mod}(#1)}

\newcommand{\sfS}{\mathsf{S}}
\newcommand{\sfT}{\mathsf{T}}

\newcommand{\sHom}{\underline{\mathrm{Hom}}}
\newcommand{\Hom}{{\mathrm{Hom}}}
\newcommand{\PHom}{{\mathrm{PHom}}}
\newcommand{\Ext}{{\mathrm{Ext}}}

\newcommand{\thick}[1]{\mathsf{Thick}\langle #1 \rangle}
\newcommand{\thickC}[2]{\mathsf{Thick}_{#1}\langle #2 \rangle}
\newcommand{\loc}[1]{\mathsf{Loc}\langle #1 \rangle}

\newcommand{\len}{\mathrm{len}}

\newcommand{\gel}{\mathrm{gel}}

\newcommand{\up}{{\uparrow}}
\newcommand{\down}{{\downarrow}}
\newcommand{\iso}{\cong}
\newcommand{\lra}{\longrightarrow}

\newcommand{\im}{\mathrm{im}}
\newcommand{\soc}{\mathrm{soc}}
\newcommand{\rad}{\mathrm{rad}}

\newcommand{\B}{\mathcal{B}}

\newcommand{\bP}{\mathbb{P}}

\newcommand{\bS}{\mathbb{S}}

\newcommand{\Z}{\mathbb{Z}}
\newcommand{\cF}{\mathcal{F}}
\newcommand{\cG}{\mathcal{G}}
\newcommand{\cI}{\mathcal{I}}

\newcommand{\cP}{\mathcal{P}}

\newcommand{\blank}{-}

\newcommand{\dfn}[1]{\textbf{#1}}

\newcommand{\xqedhere}[1]{\rlap{\hbox to#1{\hfil\llap{\ensuremath{\qed}}}}}
\usepackage[square, numbers, comma, sort&compress]{natbib}  

\usepackage{ifpdf}
\ifpdf  \usepackage[pdftex,bookmarks=false]{hyperref}
\else   \usepackage[hypertex]{hyperref}
\fi

\newcommand{\Estartt}[1]{\path (#1) node[circle, draw, inner sep=0pt, minimum width=3pt] {}}

\newcommand{\kstartt}[1]{\path (#1) node[dot] {}}

\newcommand{\eraw}[4]{+#3 node(next)[circle, draw, inner sep=0pt, minimum width=3pt] {} +(0,0) edge[#2] node[#4] {#1} (next.center) ++#3}
\newcommand{\enext}[3]{\eraw{#1}{#2}{#3}{#1}}
\newcommand{\eXo}[1]{\enext{}{post}{#1}}
\newcommand{\eYo}[1]{\enext{}{post}{#1}}
\newcommand{\exo}[1]{\enext{}{pre}{#1}}
\newcommand{\eyo}[1]{\enext{}{pre}{#1}}

\newcommand{\eXw}{\eXo{( -1,-1)}}

\newcommand{\exw}{\exo{(  1, 1)}}

\newcommand{\eYw}{\eYo{(  1,-1)}}

\newcommand{\eyw}{\eyo{( -1, 1)}}

\newcommand{\kraw}[4]{+#3 node(next)[dot] {} +(0,0) edge[#2] node[#4] {#1} (next.center) ++#3}
\newcommand{\knext}[3]{\kraw{#1}{#2}{#3}{#1}}
\newcommand{\kXo}[1]{\knext{}{post}{#1}}
\newcommand{\kYo}[1]{\knext{}{post}{#1}}
\newcommand{\kxo}[1]{\knext{}{pre}{#1}}
\newcommand{\kyo}[1]{\knext{}{pre}{#1}}

\newcommand{\kXw}{\kXo{( -1,-1)}}

\newcommand{\kxw}{\kxo{(  1, 1)}}

\newcommand{\kYw}{\kYo{(  1,-1)}}

\newcommand{\kyw}{\kyo{( -1, 1)}}

\begin{document}

\title{Computations of Generating Lengths with GAP}
\author{Gaohong Wang}
\address{Department of Mathematics\\
University of Western Ontario\\
London, ON N6A 5B7, Canada}
\email{gwang72@uwo.ca}


\begin{abstract}
%
%
In this paper, we discuss how to apply GAP to do computations in modular representation theory.
Of particular interest is the generating number of a group algebra, which
measures the failure of the generating hypothesis in the stable module category.
We introduce a computational method to do this calculation and
present it in pseudo-code.
We have also implemented the algorithm in GAP and
managed to do computations of examples that
we were not able to do before.
The computations lead to conjectures on the ghost numbers of the groups $Q_8$ and $A_4$.
\end{abstract}

\maketitle

\tableofcontents

\section{Introduction}
In this paper, we develop new algorithms to do computations in modular representation theory, and
present them in pseudo-code.
We have implemented the code in GAP~\cite{GAP}, which is a system for computational discrete algebra,
building upon the GAP package `reps' developed by Webb and others~\cite{GAP_rep},
and some extra functions written by Christensen that supplement those in the main file.
The code allows us to do computations of examples that
we were not able to do before.
And the computations lead to some conjectures in these examples.


Let $G$ be a finite group, and
let $k$ be a field whose characteristic divides the order of $G$.
We are interested in the generating number of the group algebra $kG$,
which is a numeric invariant of the stable module category $\StMod{kG}$ that
measures the failure of the generating hypothesis.
We will provide more background on the stable module category and the generating number in Section~\ref{se:background}.
Briefly speaking, the stable module category $\StMod{kG}$ is a quotient category of the module category $\Mod{kG}$,
where the projective modules are killed.
The stable module category is a triangulated category, so
we can study the generalised generating hypothesis in $\StMod{kG}$.
This is motivated by the famous conjecture by Peter Freyd in stable homotopy theory, which states that 
if a map between two compact spectra is sent to zero by the stable homotopy group functor,
then the map is null homotopic.
The conjecture is referred to as the generating hypothesis and
is still an open question.
Generalising to a triangulated category $\sfT$
together with a set of distinguished objects $\bS$,
the set of graded functors $[S,\blank]_*$ with $S \in \bS$ is analogous to
the stable homotopy group functor in the sense that
\[\text{if } [S, M]_* = 0 \text{ for all } S \in \bS
\text{ and }M \in \loc{\bS} \text{ , then }M = 0. \]
Here $[\blank, \blank]_*$ denotes the graded hom-sets in $\sfT$.
We say that $\sfT$ satisfies \dfn{the generating hypothesis with respect to $\bS$}
if the functors $[S,\blank]_*$ are faithful on $\thick {\bS}$ for all $S \in \bS$.
See Section~\ref{ss:GH} for the definition of $\thick {\bS}$ and more details .

The generalised generating hypothesis has been studied in various cases,
such as the derived category of a ring $R$ and
the stable module category of a group algebra $kG$.
For the stable module category, we take $\bS = \{ k \}$, hence
$[k,\blank]_* \iso \widehat{H}^*(G, \blank)$ is the Tate cohomology.
It is known that the generating hypothesis fails in $\StMod{kG}$ for most groups~\cite{GH for p,GH split,admit,GH per}.
In this case, we can study the degree to which the generating hypothesis fails, and
this is measured by the \dfn{generating number} of the group algebra.
We call a map in the kernel of Tate cohomology a \dfn{ghost}.
Roughly speaking, we consider the $n$-fold composite of ghosts out of a module $M$, and
the smallest integer $n$ such that each such composite is stably trivial provides
an invariant of $M$ called the \dfn{generating length} of $M$.
The \dfn{generating number} of $kG$ is defined to be the least upper bound of the generating lengths of 
modules in $\thick k$, and one can show that
the generating hypothesis holds in $\StMod{kG}$ if and only if
the generating number of $kG$ is $1$~\cite{GH split}.
This idea is formalised in a projective class, which we discuss in Section~\ref{ss:GH}.
We also show that there are equivalent characterisations of the generating length.
For example, we can consider the $n$-fold composite of \emph{universal} ghosts out of the module $M$, and
the generating length of $M$ is the smallest integer $n$ such that
the $n$-fold composite of universal ghosts out of $M$ is stably trivial.
See Section~\ref{ss:GH} for more details on universal ghosts and generating lengths.

In Section~\ref{se:gel_GAP}, we show how the idea of universal ghosts can be applied
to compute the generating length.
In general, computing the universal ghost involves modules of infinite dimension.
We prove that the generating length is the limit of
a sequence of \emph{unstable} lengths (Corollary~\ref{cor:gl_limit}).
The unstable lengths are computable using only modules of finite dimension.

In Section~\ref{se:Replace},
we introduce an algorithm for replacing a map with an injection,
which is essential for the computation of the universal ghost. 
More precisely, we can replace a map with an injection by adding a projective summand to the codomain of the map.
Since projective modules are isomorphic to zero in $\StMod{kG}$,
this replacement is equivalent to the original map.
The existing code in the extra functions computes the replacement by
adding a free module to the codomain.
As a consequence, the cokernel of the replacement can contain projective summands.
We introduce a new algorithm to do this computation and
implement it in GAP~\cite{GAP_new}.
See Section~\ref{se:GAP_ex} for examples showing that the new method is faster.
The idea is to first replace the free module by a direct sum of indecomposable projective modules.
Then we prove a condition that determines whether we need to add a map $g: M \to P$
to the original map $f: M \to N$.
More precisely, we will replace $f$ by $f + g: M \to N \oplus P$ if
the following condition is satisfied:
\[ \ker(\Hom(S, f+g)) \subsetneq \ker(\Hom(S, f)), \]
where $S = P/ \rad(P)$ is the simple module corresponding to the indecomposable projective module $P$.
Roughly speaking, we are using the fact that the map $f$ is injective
if and only if it is injective on the socle.
This can be determined by a rank computation and
is presented in pseudo-code in Section~\ref{ss:Replace}.
We also show that the method provides an optimal answer in the sense that the replacement is minimal.
This new function \texttt{ReplaceWithInj}, together with the function that computes the (unstable) universal ghost,
allows us to compute the generating length of a module within a finite range.
We present some other functions related to \texttt{ReplaceWithInj} in Section~\ref{se:Replace} as well.
For example, we need to compute the \texttt{Simple} module when
we check the condition displayed above.
And we have a dual function \texttt{ReplaceWithSurj} that replaces a map with a surjection.
We also introduce a new algorithm to compute the projective-free summand of a module in Section~\ref{ss:proj-free}, and
show that the idea in \texttt{ReplaceWithSurj} can be applied to improve the algorithm.

The generating number of a group algebra is studied for a $p$-group in~\cite{Gh in rep, Gh num} and
for a non-$p$-group in~\cite{Gh num II}, 
where theoretical and computational results are given for generating numbers and
their bounds.
In particular, we know that the generating number of $kG$ is finite, provided that
$\thick k$ is contained in the principal block $B_0$.
But when the condition $\thick k = \StMod{B_0}$ fails, we know of no examples
where we can compute the generating number or an upper bound.
It is sometimes not easy to determine whether a single module is contained in $\thick k$.
For example, we consider the group $C_3 \times S_3$ over a field $k$ of characteristic $3$ in Section~\ref{ss:S3C3},
where we compute the generating length of a $k (C_3 \times S_3)$-module and
show that it is in $\thick k$.
We also make computations for the groups $Q_8$ and $A_4$ in Section~\ref{ss:Q8},
providing evidence for the conjectures that the generating number of $Q_8$ is $3$ and
that the generating number of $A_4$ is $2$.

We give a brief summary of the contents of the paper to end the introduction:
In Section~\ref{se:background}, we provide background material for the stable module category and
the generalised generating hypothesis, and
define the generating number of a group algebra.
In Section~\ref{se:gel_GAP}, we introduce an algorithm
for computing the unstable length of a module within a finite range and
prove that the generating length of the module is the limit of the unstable lengths
as the range goes to infinity.
In Section~\ref{se:Replace}, we describe a new algorithm for the function \texttt{ReplaceWithInj} that
replaces a map with an injection and
introduce other related functions.
In Section~\ref{se:GAP_ex}, we present examples of computations with the new code,
showing that the new code is faster,
as well as providing evidence for the conjectures on the generating numbers of the groups $Q_8$ and $A_4$.

\section{Background}\label{se:background}
In this section, we review some background on
modular representation theory and 
introduce some general concepts that will be needed in the rest of the paper.

\subsection{The stable module category}
Let $G$ be a finite group and $k$ be a field whose characteristic divides the order of $G$.
The stable module category $\StMod{kG}$ is a quotient category of the module category $\Mod{kG}$.
The hom-set between two modules $M$ and $N$ in $\StMod{kG}$ is defined by
\[ \sHom(M, N) := \Hom(M, N)/ \PHom(M, N), \]
where $\PHom(M, N)$ consists of stably-trivial maps between $M$ and $N$,
i.e., the maps that factor through a projective module $P$.
Note that projective modules are isomorphic to zero in the stable module category.
To avoid ambiguity, we say that two modules $M$ and $N$ are
\dfn{stably isomorphic} if they are isomorphic in $\StMod{kG}$.
We write $\stmod{kG}$ for the full subcategory of all
finite-dimension modules in $\StMod{kG}$.
Then $\stmod{kG}$ consists of exactly the \dfn{compact} objects $M$ in $\StMod{kG}$
such that the canonical map
\[ \oplus \sHom(M, X_i) \to \sHom(M, \oplus X_i) \]
is an isomorphism for any class of objects $\{ X_i \}$ in $\StMod{kG}$.
Since the regular representation $kG$ is both projective and injective
as a module over itself,
projective and injective modules coincide in $\stmod{kG}$.
It also follows that two modules are stably isomorphic in $\stmod{kG}$
if and only if they have isomorphic projective-free summands.

The stable module category has a triangulation structure which
we now describe.
Then one will see that cohomology groups of $kG$-modules are represented by
hom-sets in the stable module category.
The desuspension $\Omega M$ of a module $M \in \StMod{kG}$ is defined to be the kernel of a surjective map $P \to M$,
where $P$ is a projective module.
Note that $\Omega M$ is well defined up to isomorphism in $\StMod{kG}$.
We write $\Omega^n M$ for the $n$-fold desuspension of $M$.
Dually, we can define $\Sigma M$ by the short exact sequence
$0 \to M \to P \to \Sigma M \to 0$,
where $P$ is a projective (and injective) module.
Now we define the group cohomology and Tate cohomology of a $kG$-module $M$.

\begin{de}
Let $G$ be a finite group and $k$ be a field.
Let \[P_*: \cdots \lra P_2 \lra P_1 \lra P_0 \]
be a projective resolution of the trivial representation $k$.
The $n$-th \dfn{group cohomology} $H^n(G,M)$ of $M$ is defined to be the $n$-th
cohomology of the chain complex $\Hom(P_*, M)$ for $n \geq 0$.

If, instead of a projective resolution, we take a complete resolution 
\[T_*: \cdots \lra P_1 \lra P_0 \xrightarrow{\partial_0} P_{-1} \lra P_{-2} \lra \cdots\]
of $k$, that is,
a doubly infinite exact sequence of projective modules such that
$\im(\partial_0) = k$, then
the $n$-th \dfn{Tate cohomology} $\widehat{H}^n(G,M)$ of $M$ is defined to be the $n$-th
cohomology of the chain complex $\Hom(T_*, M)$ for $n \in \Z$.
We can also replace the trivial module $k$ by an arbitrary $kG$-module $L$ and compute the resolutions $P_*$ and $T_*$ of $L$.
The cohomology of the chain complexes $\Hom(P_*, M)$ and $\Hom(T_*, M)$ of $M$ are denoted by $\Ext^n(L, M)$ and $\widehat{\Ext}^n(L, M)$,
respectively.
\end{de}

It is easy to see that, for $M$ and $L$ in $\StMod{kG}$ and $n \in \Z$,
there is a natural isomorphism
\[ \widehat{\Ext}^n(L, M) \iso \sHom(L, \Sigma^n M) \iso \sHom(\Omega^n L, M).\]
In particular, the Tate cohomology $\widehat{H}^n(G,M)$ of $M$
is represented by the trivial representation $k$
as $\sHom(\Omega^n k, M)$.
Moreover, by usual homological algebra,
$\widehat{\Ext}^1(L, M)$ is equivalent to
the isomorphism classes of extensions between $L$ and $M$.
Hence a short exact sequence $0 \to M \to N \to L \to 0$ corresponds to a map $\delta\in \sHom(L, \Sigma M)$.
This defines a triangle in $\StMod{kG}$:
\[ M \to N \to L \xrightarrow{\delta} \Sigma M, \]
and gives $\StMod{kG}$ a triangulation.
To compute the \dfn{cofibre} of a map $f : M \to N$,
we need to replace $f$ with an injection
that is stably equivalent to it.
For simplicity, we write $f + g$ for the map $M \to N \oplus P$, where
$f: M \to N$ and $g: M \to P$ are maps out of $M$.
If $P$ is projective, then the maps $f$ and $f+g$ are stably equivalent.
Choosing a map $g: M \to P$ such that $f+g$ is injective, then
the cofibre of $f$ in $\StMod{kG}$ is defined to be the cokernel of $f+g$.
Note again that the cofibre is well-defined up to isomorphism in $\StMod{kG}$.
Dually, we can define the fibre of a map $f$.
In Section~\ref{se:Replace}, we will present the pseudo code to compute
the replacement of a map with an injection.

\subsection{The generalised generating hypothesis and projective classes}\label{ss:GH}

In this section, we introduce the generalised generating hypothesis in a triangulated category, and
discuss its relation with a projective class.
Then we show how to apply this idea to $\StMod{kG}$.

Let $\sfT$ be a triangulated category, and
let $\bS$ be a set of distinguished objects in $\sfT$.
We write $[\blank, \blank]$ for hom-sets in $\sfT$.
A full subcategory $\sfS$ of $\sfT$ is said to be \dfn{thick} if it is closed under
suspension, desuspension, retracts, and triangles.
If in addition, $\sfS$ is closed under arbitrary sums, then
it is called a \dfn{localising} subcategory of $\sfT$.
The thick (resp. localising) subcategory generated by $\bS$ is
the smallest thick (resp. localising) subcategory that
contains $\bS$, and is denoted by $\thick{\bS}$ (resp. $\loc{\bS}$).
The set of graded functors $[S,\blank]_*$ with $S \in \bS$ is analogous to
the stable homotopy group functor in the sense that
\[\text{if } [S, M]_* = 0 \text{ for all } S \in \bS
\text{ and }M \in \loc{\bS} \text{ , then }M = 0. \]
But in general, we don't expect that $[S, \blank]_*$ is faithful
when restricted to $\loc {\bS}$, or,
in other words, $[S, \blank]_*$ will detect not only zero objects,
but also zero maps.
However, we can restrict the functors $[S, \blank]_*$ further to $\thick {\bS}$.
We say that $\sfT$ satisfies \dfn{the generating hypothesis with respect to $\bS$}
if the functors $[S,\blank]_*$ are faithful on $\thick {\bS}$ for all $S \in \bS$.
Note that if $\bS$ consists of finitely many compact objects in $\sfT$, then
\[\thick {\bS} = \loc {\bS} \cap \text{compact objects in } \sfT. \]

We write $\cI$ for the intersection of the kernels of $[S, \blank]_*$ for all $S \in \bS$. 
If $\cI$ the zero ideal, then the generating hypothesis holds.
Note that this is a stronger condition than the generating hypothesis.
Nevertheless, when the generating hypothesis fails,
the least integer $n$ such that
$\cI^n$ is zero provides some measurement of the failure of the generating hypothesis, 
where $\cI^n$ is the $n$-th power of the ideal $\cI$ that
consists of composites of $n$-fold maps in $\cI$.
We formalise this idea in the concept of a projective class:

\begin{de}
Let $\sfT$ be a triangulated category. 
A \dfn{projective class} in $\sfT$ consists of a class $\cP$ of objects of $\sfT$
and an ideal $\cI$ of morphisms of $\sfT$ such that:
\begin{enumerate}[(i)]
\item $\cP$ consists of exactly the objects $P$ such that every composite
      $P \to X \to Y$ is zero for each $X \to Y$ in $\cI$,
\item $\cI$ consists of exactly the maps $X \to Y$ such that every composite
      $P \to X \to Y$ is zero for each $P$ in $\cP$.
\item for each $X$ in $\sfT$, there is a triangle $P \to X \to Y \to \Sigma P$ with 
      $P$ in $\cP$ and $X \to Y$ in $\cI$.
\end{enumerate}
\end{de}
\noindent Note that $\cP$ is closed under retracts and arbitrary direct sums.
If $\cP$ (or equivalently $\cI$) is closed under suspension and desuspension,
then we say that the projective class $(\cP, \cI)$ is \dfn{stable}.
The map $X \to Y$ in the third condition is a universal map out of $X$ in $\cI$.
In general, for a class of objects $\cP$,
we can define a nested sequence of classes by
\begin{enumerate}[(i)]
\item $\cP_1 = \cP$, and
\item $X \in \cP_n$ if $X$ is an retract of some object $M$
such that $M$ sits in a triangle
$P \to M \to Q$ with $P \in \cP$ and $Q \in \cP_{n-1}$.
\end{enumerate}
For an object $X$ in $\sfT$, the \dfn{length $\len(X)$} of $X$ with respect to $(\cP, \cI)$
is defined to be the smallest integer $n$ such that $X \in \cP_n$,
if such an $n$ exists.
There is an alternative interpretation of $\len(X)$ using $\cI^n$ by the property of a projective class, which
we state as the next lemma.
By convention, $\cP_0$ consists of all zero objects in $\sfT$ and
$\cI^0$ consists of all maps in $\sfT$.

\begin{lem}[{\cite{Chr}}]\label{le:uni-map}
Let $\sfT$ be a triangulated category, and $(\cP, \cI)$ be a (possibly unstable) projective class in $\sfT$.
Then, for all integers $n \geq 0$, $(\cP_n, \cI^n)$ is a projective class in $\sfT$. 
In particular, the following conditions are equivalent for an object $X$ in $\sfT$:
\begin{enumerate}[(i)]
\item $X$ is in $\cP_n$.
\item Every $n$-fold composite of maps in $\cI$ out of $X$ is zero.
\item The $n$-fold composite of universal maps in $\cI$ out of $X$ is zero.
\end{enumerate}
\end{lem}

Now we consider $\StMod{kG}$ and the Tate cohomology functor.
We call a map in the kernel of the Tate cohomology functor a \dfn{ghost} and
write $\cG$ for the ideal of ghosts in $\StMod{kG}$.
Let $\cF$ be the class of objects in $\StMod{kG}$ generated by the trivial representation $k$
under retracts, arbitrary direct sums, suspension and desuspension.
Since the Tate cohomology is represented by $k$,
the pair $(\cF, \cG)$ forms a projective class in $\StMod{kG}$, and
this is called the \dfn{ghost projective class}.
For $M \in \thick k$, the \dfn{generating length} $\gel(M)$ of $M$ is
the length of $M$ with respect to $(\cF, \cG)$.
The generating number of $kG$ is defined to be the least upper bound of $\gel(M)$ for
all $M \in \thick k$.

There is another invariant called the \emph{ghost number} that is more closely related to
the generating hypothesis in $\StMod{kG}$.
In general, the ghost number of $kG$ is less than or equal to the generating number.
But in the examples that we are able to compute,
we have shown them to be equal.
We will focus on the computation of the generating number in this paper.
See~\cite{Gh in rep, Gh num, Gh num II} for further discussions on the difference
between the ghost number and the generating number.

\section{A computational method to calculate the generating length}\label{se:gel_GAP}

By Lemma~\ref{le:uni-map}, the generating length of a module $M$ can be computed using
universal ghosts.
The idea is presented in the following pseudo-code:
\begin{verbatim}
    LengthHelper = function with inputs:
                   a map f from M to N and an integer n
        g = universal ghost from N to L

        if f composed with g is stably trivial then
            return f and n
        return LengthHelper(f composed with g, n+1)

    Length of M = LengthHelper(the identity map on M, 1)
\end{verbatim}

However, computing the universal ghost involves modules of infinite dimension,
unless the Tate cohomology of $M$ is finitedly generated.
Hence, to make this idea work, we need to first consider unstable ghosts within a finite range.
Let $\cF(-m, m)$ be the class 
generated by $\{ \Sigma^i k \st -m \leq i \leq m \} \subseteq \cF$
under retracts and arbitrary direct sums.
Then $\cF(-m, m)$ forms part of a projective class, and
the relative null maps consists of the unstable ghosts within the range
$[-m, m]$.
Given $M \in \thick k$,
we write $\gel_m(M)$ for the length of $M$ with respect to $\cF(-m, m)$.
Since $\cF(-m, m) \subseteq \cF(-m-1, m+1) \subseteq \cdots \subseteq \cF$,
we get a decreasing sequence greater than or equal to $\gel(M)$:
\[\gel_m(M) \geq \gel_{m+1}(M) \geq \cdots \geq \gel(M), \]
where $\gel_m(M)$ can be computed using the pseudo-code presented above for each $m \geq 0$.

\begin{ex}
Let $G$ be a $p$-group and let $k$ be a field of characteristic $p$.
Let $M$ be a projective-free $kG$-module.
Then $\gel_0 (M)$ is
equal to the radical length of $M$~\cite[Proposition~4.5]{Gh num}.
\end{ex}

\begin{ex}	
If the cohomology of $kG$ has periodicity $n$, 
then $\gel(M) = \gel_{\lfloor \frac{n}{2} \rfloor} (M)$ for all $M \in \thick k$, and
the computation of the generating length of $M$ is a finite process.
\end{ex}

Now we want to show that the limit of $\gel_m(M)$ is $\gel(M)$.
This will be a corollary of Lemma~\ref{le:BoVdB}.
We need some more notations before we introduce the lemma.
Let $\sfT$ be a triangulated category with compact objects $\sfT^c$.
Let $\bP$ be a set of compact objects in $\sfT$.
We write $\cP$ for the class of objects generated by $\bP$ under retracts, arbitrary direct sums, suspension, and desuspension and
write $\cP^c$ for the class of objects generated by $\bP$ under retracts, finite sums, suspension, and desuspension.
Note that $\cP^c = \cP \cap \sfT^c$.
More generally, we can define $\cP^c_n := (\cP^c)_n$ in a similarly pattern as $\cP_n$, and
the following lemma holds:
\begin{lem}[{\cite[Proposition 2.2.4]{BoVdB}}]\label{le:BoVdB}
Let $\sfT$ be a triangulated category, and let $\bP$ be a set of compact objects in $\sfT$.
With the notation introduced above,
\[ \cP^c_n = \cP_n \cap \sfT^c. \]
In particular, $\thick {\bP}$ = $\loc {\bP} \cap \sfT^c$.
Moreover, the sequence
\[ \cP^c_1 \subseteq \cP^c_2 \subseteq \cdots \subseteq \cP^c_n \subseteq \cdots \subseteq \thick {\bP}\]
is a filtration of $\thick {\bP}$ with $\thick {\bP} = \bigcup \cP^c_n$.\qed
\end{lem}

As a corollary, we can compute the generating length in $\stmod{kG}$.
\begin{cor}\label{cor:gl_limit}
Let $G$ be a finite group and $k$ be a field whose characteristic divides the order of $G$.
Let $M$ be a module in $\thick k$. Then
$\gel(M) = \lim_{m \to \infty} \gel_m (M)$.
\end{cor}

Since $(\gel_m(M))$ is a sequence of integers, we conclude that $\gel_m (M) = \gel(M)$ for $m$ large.

\begin{proof}
Consider $\bP = \{ k \}$ in $\StMod{kG}$.
Let $M$ be a module in $\thick k$.
It follows from the lemma that $M \in \cP_n^c$, with $n = \gel(M)$.
However, there are only finitely many spheres $\Sigma^{n_i} k$ needed
to build up $M$ in $n$ steps.
Hence there exists an integer $m$, such that
$M \in (\cF(-m, m))_n$, and $\gel_m (M) \leq n = \gel(M)$.
Conversely, since $\cF(-m, m)$ is contained in $\cF$, $\gel(M) \leq \gel_m(M)$.
\end{proof} 


\begin{rmk}
We remark here that there is not a universal choice of $N$ such that $\gel_N(M) = \gel(M)$ for all $M \in \thick k$.
Indeed, if the group cohomology is not periodic, then
$\gel_N( \Omega^n k) = \gel(\Omega^n k)$ if and only if $N \geq |n|$, and
the number $N$ can be arbitrarily large.
Note that the numbers $\gel_n(M)$ give upper bounds of the generating length of $M$.
Hence if a lower bound of the generating length of $M$ is known, we can hope to get the exact answer
of the generating length of $M$.
It would also be interesting to know whether there is a way to
compute lower bounds for the generating length that converge to the correct answer.
\end{rmk}

\section{New algorithms developed for computations in $\StMod{kG}$}\label{se:Replace}
We have improved the GAP code used in the `reps' package~\cite{GAP_rep} to
compute the universal ghost and generating length.
We introduce the function \texttt{ReplaceWithInj} in this section,
which is essential for computing the universal ghost.
We also show the relation of \texttt{ReplaceWithInj} with other functions.

\subsection{The ReplaceWithInj function and the Simple function}\label{ss:Replace}
Recall that the universal ghost is the cofibre of a map that is surjective
on Tate cohomology, and
computing the cofibre depends on a function that replaces a map with an injection
that is stably equivalent to it.
For simplicity, we write $f + g$ for the map $M \to N \oplus P$, where
$f: M \to N$ and $g: M \to P$ are maps out of $M$.
If $P$ is projective, then the maps $f$ and $f+g$ are stably equivalent.
Now let $\{P_i\}$ be the set of non-isomorphic indecomposable projective $kG$-modules, and
let $\B_i$ be a basis for $\Hom(M, P_i)$.
Observe that the natural map
\[ \alpha: M \to \oplus_i(\oplus_{g \in \B_i} P_i)\]
is injective.
Then for any map $f: M \to N$, the map $f + \alpha$ is a replacement of $f$ with an injection.
But in this way, we will have added more maps than we need to the map $f$.
For example, we don't need the maps $g$ with $\ker(f+g) = \ker(f)$.
In fact, we can do better than this and get rid of more maps that we don't want.
We need a lemma before we state the condition that we will put on $g$.

\begin{lem}\label{le:Inj}
Let $f: M \to N$ be a map in $\modu{kG}$.
Then the map $f$ is injective if and only if,
for any simple module $S$, the map
\[ \Hom(S, f): \Hom(S, M) \to \Hom(S,N) \]
is injective.
\end{lem}

\begin{proof}
Since $\ker(\Hom(S, f)) \iso \Hom(S, \ker(f))$, the map $f$ being injective implies that
$\Hom(S, f)$ is injective for any $S \in \modu{kG}$.
Conversely, if $\Hom(S, \ker(f)) = 0$ for all simple modules, then
$\ker(f) = 0$
because the simple modules generate the module category.
\end{proof}

It follows from the lemma that
we only need to add to $f$ those maps $g$ that shrink $\ker(\Hom(S, f))$ for some simple module $S$.
Recall that, for each indecomposable projective module $P$,
there is a simple module corresponding to it, given by $P / \rad(P)$:

\begin{lem}[{\cite[Theorem~1.6.3]{Benson}}]\label{le:simple-proj}
Let $P$ be an indecomposable projective $kG$-module.
Then the radical quotient $P / \rad(P)$ is simple and $P / \rad(P) \cong \soc(P)$.
Moreover, the assignment of $P/ \rad(P)$ to $P$ provides a one-one correspondence
between isomorphism classes of indecomposable projective $kG$-modules and
simple $kG$-modules. \qed
\end{lem}

Now let $P$ be an indecomposable projective module, and
let $g$ be a map from $M$ to $P$.
We claim that, to decide whether we need to replace $f$ by $f+g$,
it suffices to check the condition
\begin{equation}\label{eq:replace}
\ker(\Hom(S, f+g)) \subsetneq \ker(\Hom(S, f)),
\end{equation}
for $S = P/ \rad(P)$.
Indeed, if $S' \ncong S$ is another simple module, then
$\Hom(S', P) = 0$, and
since $\ker(\Hom(S, f+g)) = \ker(\Hom(S, f)) \cap \ker(\Hom(S, g))$,
there is no need to check $g$ on $S'$.

It follows the discussion above that we can work with one indecomposable projective $P$ at a time.
Observe that if we have replaced $f$ with $f' = f + g$, then
we can replace the condition in Equation~\ref{eq:replace} with
$\ker(\Hom(S, f'+g)) \subsetneq \ker(\Hom(S, f'))$.
Also note that if $\{g_1, g_2, \dots, g_l\}$ is a basis for $\Hom(M, P)$, then
\[ \ker(\Hom(S, \sum_{i=1}^l(g_i))) = \ker(\Hom(S, \alpha)) = 0, \]
where $\alpha: M \to \oplus_i(\oplus_{g \in \B_i} P_i)$ is the injection we started with.
Hence, the following pseudo-code produces a replacement $f'$ of $f$ such that $\ker(\Hom(S, f')) = 0$ for $S = P/ \rad(P)$:
\begin{verbatim}
    ReplaceWithInj = function with one input: a map f from M to N
        P = an indecomposable projective module
        S = the simple module corresponding to P
        for g in a basis for Hom(M, P)
            if ker(Hom(S, f+g)) is strictly contained in ker(Hom(S, f)) then
                replace f with f+g
                continue the loop of g until ker(Hom(S, f)) = 0
        return f
\end{verbatim}
Then, by Lemma~\ref{le:Inj}, we can loop the preceding process
over all indecomposable projective modules and
produce a replacement by an injection.

\begin{rmk}\label{rm:jointly-surj}
Note that the `for' loop of $g$ over $\Hom(M, P)$ can be replaced
any set of maps $\{g_1, g_2, \dots, g_l\}$ in $\Hom(M, P)$, such that
\[ \ker(\Hom(S, \sum_{i=1}^l(g_i))) = 0. \]
In particular, if the injective hull of $M$ has been computed,
we can use it when we compute the replacement of a map $f: M \to N$ with an injection.
\end{rmk}

Now we describe how to check the condition whether
$\ker(\Hom(S, f+g)) \subsetneq \ker(\Hom(S, f))$.
This is done by a rank computation.
We form the map $\beta: \oplus S \to M$, where
the sum ranges over a basis for $\Hom(S, M)$.
Then we compare the dimensions of $\im((f+g) \circ \beta)$
and $\im(f \circ \beta)$ in the diagram
\[
\xymatrix{
                         &                           & N \oplus P \ar[dr] \\
 \oplus S \ar[r]^{\beta} & M \ar[rr]^f \ar[ur]^{f+g} &                    & N.
}
\]
It is clear that $\mathrm{rank}((f+g) \circ \beta) \geq \mathrm{rank}(f \circ \beta)$.
Since $\oplus S$ is semi-simple,
the equality holds if and only if
$\ker(\Hom(S, f+g)) = \ker(\Hom(S, f))$.
In other words, the following conditions are equivalent:
\begin{enumerate}
\item $\ker(\Hom(S, f+g)) \subsetneq \ker(\Hom(S, f))$,
\item $\mathrm{rank}((f+g) \circ \beta) > \mathrm{rank}(f \circ \beta)$.
\end{enumerate}
Note that $\mathrm{rank}(f \circ \beta)$ is at most $\mathrm{rank} (\beta)$, and
this is equivalent to $\ker(\Hom(S, f))=0$, so
we can break out the loop over the basis for $\Hom(M, P)$ when
$\mathrm{rank}(f \circ \beta) = \mathrm{rank}(\beta)$.
We can also check at the same time whether $f$ is injective or not and, if yes,
we return $f$ to avoid the extra loop over the other projective modules.
To conclude the discussion, we display the function ``ReplaceWithInj'' in the following pseudo-code:

\begin{verbatim}
    ReplaceWithInj = function with one input: a map f from M to N
        f = a given map from M to N
        if Rank(f) == dimension of M then   % f is injective
            return N and f
        L = list of non-isomorphic indecomposable projectives
        for P in L
            S = the simple module corresponding to P
            b = map from a sum of S to M, ranging over a basis for Hom(S, M)    
            r = Rank(f composed with b)
            rankb = Rank(b)    
            if r !== rankb then 
                % r not maximal, so need to loop over a basis for Hom(M, P)
                for g in a basis for Hom(M, P)
                    newf = f + g
                    newr = Rank(newf composed with b)
                    if newr > r then
                        f = newf
                        r = newr
                        N = direct sum of N and P
                    if r == rankb then         % r is maximal
                        if Rank(f) == dimension of M then
                            return N and f
                        break out of the loop over the basis for Hom(M, P)
        % This point should never be reached
        return
\end{verbatim}

\begin{rmk}
The code produces an optimal answer in the sense that
the replacement is minimal, unless
the map $f$ itself contains a stably trivial summand,
in which case we need to exclude the summand.
In particular, if $N$ is the zero module, then
we will compute the injective hull of $M$.

To see that the process is optimal,
observe first that $\ker((f+g)\circ \beta) \subseteq \ker(f \circ \beta)$ is the kernel of the composite
\[ \ker(f \circ \beta) \to \oplus S \xrightarrow{\beta} M \xrightarrow{g} P. \]
Since $\ker(f \circ \beta)$ is a direct sum of copies of the simple module $S$ and
$P$ is the corresponding projective module,
the image of this composite is either zero or isomorphic to $S$.
It follows that, when we replace $f$ by $f+g$, we always have
\[\mathrm{rank}((f+g)\circ \beta) = \mathrm{rank}(f \circ \beta) + \dim(S). \]
Thus, to replace a map $f: M \to N$ by an injection,
we need to add exactly \[\frac{\mathrm{rank}(\beta) - \mathrm{rank}(f \circ \beta)}{\dim(S)}\] copies
of the projective module $P$ to $N$, as
our code will do.
Since this number is independent of the choice of a basis for $\Hom(M, P)$,
the code is optimal.
\end{rmk}

Note that the algorithm we introduced depends on a decomposition function to
find all indecomposable projective modules and,
for each indecomposable projective module,
we need to find the corresponding simple module $S = P/ \rad (P)$.

To find $S = P/ \rad (P)$, observe that  by Lemma~\ref{le:simple-proj},
there is a self map on $P$
\[ f: P \to P/ \rad(P) \iso \soc(P) \to P ,\]
with $\im(f) \iso S$.
Hence we can compute the image of \emph{all} self maps on $P$
to find $S$ as the image whose dimension is the smallest, but
this is not very efficient.
So we replace $P$ with $M = \im(f)$, where $f$ is a self map on $P$.
Since $M$ is both a submodule and a quotient module of $P$,
it also satisfies the condition that $M / \rad(M) \iso \soc(M) \iso S$.
Then we can find $S$ as the image of a self map on $M$.
To implement this idea, we can loop over all self maps $f$ on $P$ and
compute $M = \im(f)$.
Then, if $M$ is a proper submodule of $P$,
we replace $P$ with $M$ and make a recursive call and compute the images of self maps on $M$.
The recursion will end with a module $S$ that has no proper submodules.
In other words, $S$ is simple.
Note that if $\Hom(M, M)$ has dimension $1$ and
$M / \rad(M) \iso \soc(M)$, then
the map $M \to M/ \rad(M) \iso \soc(M) \to M$ is an isomorphism, hence
$M$ is simple, and
we can return $M$ in this case.
In conclusion, if $P$ is an indecomposable projective module, then
we can find the corresponding simple module $S$ with the following pseudo-code:
\begin{verbatim}
    Simple = a function with one input: a kG-module P such that
                                        P/rad(P) is isomorphic to soc(P)

        hom = Hom(P, P)
        if hom has dimension 1 then
            return P
        for all maps f in hom
            if 0 < Rank(f) < dimension of P then
                return Simple(im(f))
        % This point can be reached when k is not algebraically closed
        return P
\end{verbatim}

\begin{rmk}
Note that not every simple module $S$ has $\dim(\Hom(M, M)) = 1$
when the field $k$ is small.
So, in general, we have to search over all self maps on $M$.
Also note that, for an arbitrary module $M$,
$\dim(\Hom(M, M)) = 1$ does not imply that $M$ is simple.
For a counterexample, take $G = S_3$, the symmetric group on three letters, and
consider the two dimensional module $M = \widetilde{\Omega} k$,
where the condition $M/ \rad(M) \iso \soc (M)$ fails.
However we have seen that the condition always holds
for the module $M$ that arises in this algorithm.
\end{rmk}



\subsection{Other functions related to ReplaceWithInj}\label{ss:gl_GAP}
In this section, we show the relation of the function \texttt{ReplaceWithInj} 
with other functions.
\begin{enumerate}
\item \texttt{Cofibre} and \texttt{Suspension}.

With the \texttt{ReplaceWithInj} function, we can compute the cofibre of a map $f$.
In particular, replacing the zero map out of $M$,
we get the injective hull of $M$, and
its cofibre is the suspension of $M$.
Since the \texttt{ReplaceWithInj} function provides an optimal answer,
the suspension of $M$ we get is projective-free.
\texttt{Cofibre} is also essential in the \texttt{Length} function,
where we need to compute universal ghosts.

\item \texttt{CreateRandomModule}.

We can create random modules in $\thick k$ using cofibres.
We choose a random map $f: P \to Q$ between random modules $P$ and $Q$ that are sums of suspensions and desuspensions of $k$ and
compute the cofibre $R_1$.
Note that $R_1$ has generating length at most $2$.
Iterating the process $n$-times, we can build up a module $R_n$ of length at most $n+1$.
Note that the function depends on the number of summands that we allow in each step and
the number of steps $n$ that we take.

\item \texttt{IsStablyTrivial}.

Let $f: M \to P$ be an injection of $M$ into a projective module.
Then since $P$ is also injective, every map from $M$ to a projective module
factors through $f$.
Hence \texttt{ReplaceWithInj} provides an algorithm to detect whether
a map $g: M \to N$ is stably-trivial or not,
by checking whether it factors through $f$.

\item \texttt{ReplaceWithSurj}, \texttt{Fibre} and \texttt{Desuspension}.

Since the pseudo-code we present in \texttt{ReplaceWithInj} is dualizable, 
we can write the dual functions \texttt{ReplaceWithSurj},
\texttt{Fibre} and \texttt{Desuspension}.
\end{enumerate}

\subsection{The ProjectiveFreeSummand function}\label{ss:proj-free}

We introduce a new algorithm to compute the projective-free summand of a $kG$-module $M$, and
show that the idea in Section~\ref{ss:Replace} can be applied to improve the algorithm.
The existing code for computing the projective-free summand 
first computes the indecomposable summands of $M$, and then
tests each of these summands and excludes the projective ones.
This consumes more memory and time.
The new algorithm will also need to decompose the regular representation once in order to
find all indecomposable projective $kG$-modules, but
it appears to be significantly faster than the old one.
See the next section for an example that compares the time needed
for the different algorithms for computing the projective-free summand.


Let $f_i: P_i \to M$ be a set of maps that is jointly surjective,
with each $P_i$ being indecomposable and projective, and
let $f: N \to M$ be a map to $M$.
Recall that we write $f + f_i$ for the map $N \oplus P_i \to M$ that is $f$ on $N$ and $f_i$ on $P_i$.
We can compute the projective-free summand of $M$ by the following algorithm:
\begin{verbatim}
    ProjectiveFreeSummand = function with one input: a module M
        f_i = a set of maps from P_i to M that is jointly surjective, with
              each P_i being indecomposable and projective
        f = zero map from zero module to M
        r = 0
        for each f_i
            newf = f + f_i
            newr = Rank(newf)
            if newr == r + dimension of P_i then
                f = newf
                r = newr
        return quotient module of M by the image of f
\end{verbatim}
By construction, the image of $f$ is a summand of $M$ that is projective.
On the other hand, if $P$ is an indecomposable projective summand of $M$,
then there exists some $f_i$ such that $f_i$ maps isomorphically onto $P$.
By induction on $M/P$, one can show that the image of $f$ finally becomes the projective summand of $M$.
The algorithm works for any set of maps $f_i: P_i \to M$ that is jointly surjective with
each $P_i$ being indecomposable and projective.
Since the projective modules $P_i$ are required to be indecomposable,
we need to call the \texttt{Decompose} function here to find them.
And this is the only place that we need to use \texttt{Decompose}.

There are different ways to get the maps $f_i$.
The intuitive idea will be computing a basis for $\Hom(P_i, M)$ for each indecomposable projective $P_i$.
Or we can use the projective cover of $M$ here,
which can be computed by \texttt{ReplaceWithSurj}.
This idea can reduce the number of rank computations,
but we pay the cost of checking more conditions in the loops and doing more matrix multiplications.
However, there will be potential savings in time as we apply this idea and
avoid the unneeded loops.
We have implemented the latter algorithm in GAP and compared it with the existing algorithm.
The results will be presented in the next section.

\section{Examples}\label{se:GAP_ex}
In this section, we give examples of computations with the new code.
We compare the new code with the old code in Section~\ref{ss:compare}, and
show that the new code is faster in computing suspensions and desuspensions.
Then we make computations for the groups $Q_8$ and $A_4$ in Section~\ref{ss:Q8},
providing evidence for the conjectures that the generating number of $Q_8$ is $3$
and that the generating number of $A_4$ is $2$.
And in Section~\ref{ss:S3C3},
we make computations for the group $C_3 \times S_3$, where
$\thick k \neq \StMod{B_0}$.

\subsection{Comparing the new code with the old code}\label{ss:compare}
As a special example of fibres and cofibres,
we begin with an easy computation of suspensions and desuspensions of the trivial representation
for the alternating group $A_4$ over the field $GF(4)$, and
compare the time used by the different versions of the functions \texttt{Suspension} and \texttt{Desuspension}.
We iterate \texttt{Suspension} or \texttt{Desuspension} to compute $\Sigma^{n} k$ and measure the total time used.
\begin{center}
\begin{tabular}{| c | c | c | c | c | }
\hline
$\Sigma^n(k)$ & \multicolumn{2}{|c|}{$n=50$} & \multicolumn{2}{|c|}{$n=-50$}\\\hline
                 & Dimension & Time   & Dimension & Time \\\hline
new function                     & 101        & 5.1s   & 101 & 5.1s \\\hline
old function                     & 109        & 34.8s  & 109 & 31.1s \\\hline
\end{tabular}
\end{center}
Since the old function adds free summands to the target to replace a map with an injection,
and similarly for replacing a map with a surjection,
the replacement we get by using the old code can fail to be minimal for non-$p$-groups,
which produces projective summands in the answer.
In the example, it raises the dimension of $\Sigma^{\pm 50}k$ by $8$.
To get the optimal answer using the old function,
there is an extra step to determine the projective-free summand,
while we have shown that the new algorithm always produces an optimal answer.
It is also clear from the table that the new code is faster than the old code.

Now we compare the time needed for the different versions of \texttt{ProjectiveFreeSummand}
to compute the projective-free summands of $\Sigma^{30} k$ and $\Sigma^{31} k$.
We take smaller modules here to reduce the time for the tests.
Note that the dimensions of $\tilde{\Sigma}^{30} k$ and $\tilde{\Sigma}^{31} k$ are $61$ and $63$, respectively.
We pre-compute the modules $\Sigma^{30} k$ and $\Sigma^{31} k$ with the old function, and
take the result as our models.
The old function returns a module of dimension $61$ for $\Sigma^{30} k$,
which is actually projective-free.
But it returns a module of dimension $71$ for $\Sigma^{31} k$,
which contains a projective summand of dimension $8$.
Indeed, the old function first finds a free cover $f: F \to \Sigma^{30} k$ of $\Sigma k$ and then
computes $\Sigma^{31} k$ as $\ker(f)$.
In our example, $F$ has dimension $132$ and is the minimal free cover of $\Sigma^{30} k$.
Hence, a module of dimension $71$ is the best answer we can get for $\Sigma^{31} k$ using the old function in this case. 
The following table shows the time needed in computing the projective summands of
the pre-computed modules $\Sigma^{30} k$ and $\Sigma^{31} k$:

\begin{center}
\begin{tabular}{| c | c | c | }
\hline
\texttt{ProjectiveFreeSummand} & $\Sigma^{30} k$, dimension $61$ & $\Sigma^{31} k$, dimension $71$\\\hline
                & Time   & Time \\\hline
new function    & 0.11s  & 0.17s \\\hline
old function    & 52.1s  & 71.0s \\\hline
\texttt{Decompose}    & 52.0s  & 70.8s \\\hline
\end{tabular}
\end{center}
Recall that the old \texttt{ProjectiveFreeSummand} function first
decomposes a module into the sum of its indecomposable summands and
then excludes the summands that are projective.
The last line in the table shows the time spent to decompose the module in the old method for
computing the projective-free summand.
It shows that decomposing the module is the dominant part of the old method.
Even for the module $\Sigma^{30} k$, which is projective-free,
it takes a long time for the computer to check with the old code that
it does not contain a projective summand.
One also sees clearly from the table that the new function for computing the projective-free summand is significantly faster.

For a $p$-group, since the regular representation is indecomposable,
the old function generally produces an optimal answer.
But the new \texttt{Suspension} function is still faster in this case, as
one can see in the following table, where
we compute $\Sigma^{\pm 50} k$ for the group $C_3 \times C_3$ over the field $GF(3)$:
\begin{center}
\begin{tabular}{| c | c | c | c | c | }
\hline
$\Sigma^n(k)$ & \multicolumn{2}{|c|}{$n=50$}  & \multicolumn{2}{|c|}{$n=-50$}\\\hline
                 & Dimension & Time   & Dimension & Time \\\hline
new function                     & 226        & 19.2s   & 226 & 19.9s \\\hline
old function                     & 226        & 147.2s  & 226 & 117.5s \\\hline
\end{tabular}
\end{center}
Note that it is not guaranteed by the old algorithm that the answer is going to be optimal,
even for a $p$-group.
Also note that it takes more time for the old function to compute $\Sigma^{50} k$
than to compute $\Sigma^{-50} k$
because the old code needs more time to find an injection from a module $M$ into a free module
in order to compute $\Sigma M$.

\subsection{Computations in $C_9$, $Q_8$, and $A_4$}\label{ss:Q8}
We test our code for the cyclic group $C_9$ of order $9$ with $k = GF(3)$,
the quaternion group $Q_8$ of order $8$ with $k = GF(2)$, and
the alternating group $A_4$ of order $12$ with $k = GF(4)$.
Note that the cohomology of $C_9$ has periodicity $2$ and that
the cohomology of $Q_8$ has periodicity $4$,
so we can compute the generating lengths of $kC_9$ and $kQ_8$-modules exactly in these cases.
Recall that the generating number of $kC_9$ is $4$,
the generating number of $kQ_8$ is $3$ or $4$, and
the generating number of $kA_4$ is $2$, $3$ or $4$~\cite{Gh in rep, Gh num, Gh num II}.
In the following examples, we will create modules using the function \texttt{CreateRandomModule} introduced in Section~\ref{ss:gl_GAP}, and
keep the cofibres $R_n$ with $n \geq 3$,
so that $R_n$ can have lengths greater than or equal to $4$.
Then we compute their generating lengths.

For the group $C_9$, we first record the dimensions and lengths of $R_3$ and $R_4$.
We performed $6$ trials and get

\begin{center}
\begin{tabular}{| c | c | c || c | c || c | c || c | c || c | c || c | c |}
\hline
$n$        & 3 & 4& 3 & 4& 3 & 4& 3 & 4& 3 & 4& 3 & 4 \\\hline
Dimension           & 17 & 22 & 30 & 29 & 17 & 8 & 22 & 15 & 7 & 15 & 7 & 16 \\\hline
Length   &  1 &  2 &  2 &  3 &  1 & 1 &  2 &  2 & 3 &  4 & 2 &  2 \\\hline
\end{tabular}
\end{center}
The process seldom produces a module that achieves that generating number $4$.
But when we take larger $n$, we find see more $kC_9$-modules of length $4$:
\begin{center}
\begin{tabular}{| c | c | c | c | c | c | c | c | c | c | c | c | c | c | c | c |}
\hline
$n$        & 3 & 4& 5 & 6 &7 &8 &9 &10 &11 &12 &13 &14 &15 &16 &17 \\\hline
Dimension         & 22 & 14 & 20 & 19 & 11 & 11 & 11 & 12 & 11 & 19 & 18 & 18 & 8 & 16 & 9 \\\hline
Length            &  2 & 2  & 3  & 4  & 4  & 4  & 4  & 4  & 3  & 3  & 3  & 3  & 3 & 3  & 2 \\\hline
\end{tabular}
\end{center}
It is interesting to note that
the lengths can decrease in a single trial as we take more steps to build up the modules.
Now we repeat the trial many times and check the number of appearances of the modules of
different generating lengths.
The following table is the result we get from a total of $100$ trials:

\begin{center}
\begin{tabular}{| c | c | c | c | c | c | c | c | c | c | c | c | c | c | c | c | c |}
\hline
$n$      & 2 & 3 & 4& 5 & 6 &7 &8 &9 &10 &11 &12 &13 &14 &15 &16 &17 \\\hline
Length = $1$         & 72   & 45 & 30  & 26  & 22  & 14  & 12  & 15  & 11  & 10  & 10  & 10  & 9  & 13 & 12  & 10 \\\hline
Length = $2$         & 28  &  44 & 47  & 34  & 35  & 36  & 38  & 30  & 29  & 29  & 28  & 26  & 25  & 21 & 20  & 23 \\\hline
Length = $3$         & 0  &  11 & 19  & 29  & 22  & 22  & 22  & 25  & 31  & 29  & 30  & 28  & 31  & 24 & 28  & 32 \\\hline
Length = $4$         & 0  &  0 & 4  & 11  & 21  & 28  & 28  & 30  & 29  & 32  & 32  & 36  & 35  & 42 & 40  & 35 \\\hline
\end{tabular}
\end{center}
We can see that, for $n = 2$, we only get modules of lengths less than or equal to $2$, and
similarly for $n = 3$.
As $n$ grows larger, we start to see modules of greater lengths, and
the distribution of modules of different lengths becomes quite steady for $n \geq 10$, which
resembles the behaviour of a Markov chain.
We also see that the modules of top lengths appear at a quite high frequency.

We have performed many more trials for $C_9$ and see this pattern show up again.
But this is a very special example with the group being a cyclic $p$-group.
In general, it is an interesting question to see whether
there is a similar pattern for any finite group.

Now we apply the method to study $kQ_8$-modules.
In this case, we are looking for a $kQ_8$-module of length $4$.
It would imply that the generating number of $kQ_8$ is $4$.
We have tried to build up $kQ_8$-modules with $n$ up to $100$, but
in all the examples, there are no $kQ_8$-modules of length $4$,
strongly suggesting that the generating number of $kQ_8$ is $3$.

\begin{conj}\label{conj:Q_8}
Let $G = Q_8$ and $k$ be a field of characteristic $2$. Then
\[ \text{generating number of } kQ_8 = 3.\]
\end{conj}
For evidence, here is the result when we built up $kQ_8$-modules with $n=10$.
We allowed up to $5$ summands in each step to build up the modules and
performed a total of $200$ trials:
\begin{center}
\begin{tabular}{| c | c | c | c | c | c | c | c |}
\hline
$n$      & 4& 5 & 6 &7 &8 &9 &10 \\\hline
Length = $1$         & 3   & 1 & 1  & 1  & 2  & 0  & 1\\\hline
Length = $2$         & 46   & 36 & 26  & 28  & 19  & 24  & 25\\\hline
Length = $3$         & 151   & 163 & 173  & 171  & 179  & 176  & 174\\\hline
\end{tabular}
\end{center}
We have not included the line with $\text{Length} = 4$, since
we never encountered a $kQ_8$-module with generating length $4$, which
would have disproved the conjecture.

Similarly, we have built up $kA_4$-modules with $n = 10$ and up to $5$ summands.
The modules all have length $2$,
making us believe that the generating number of $kA_4$ is $2$.
\begin{conj}\label{conj:A_4}
Let $G = A_4$ and $k$ be a field of characteristic $2$. Then
\[ \text{generating number of } kA_4 = 2.\]
\end{conj}

\subsection{The group $C_3 \times S_3$ at the prime $3$}\label{ss:S3C3}

We know from~\cite[Theorem~4.7]{Gh num II} that 
if the thick subcategory $\thick k$ generated by the trivial representation $k$
in $\stmod{kG}$ consists of all the modules in the principal block,
that is,
\begin{equation}\label{eq:thick=B0}
\thick k = \stmod {B_0},
\end{equation}
then the ghost number of the group algebra $kG$ is finite.
In general, when condition~\ref{eq:thick=B0} fails,
we don't know whether the ghost number of $kG$ is finite or not.
In this case, we can show that a module $M$ is in $\thick k$ by
showing that it has finite generating length.
We make computations for the group $C_3 \times S_3$ in this section, where
condition~\ref{eq:thick=B0} fails.

Let $G = C_3 \times S_3$ be the direct product of
the cyclic group $C_3$ of order three and the symmetric group $S_3$ on three letters.
Let $k$ be a field of characteristic $3$.
We write $x$ for a generator of $C_3$, $y = (1 2 3)$ for an element of order $3$ in $S_3$ and
$z = (1 2)$ for an element of order $2$ in $S_3$.
Thus $G$ is a group on three generators $x$, $y$, and $z$
subject to the relations $x^3=y^3=z^2=1$, $xy=yx$, $xz=zx$, and $yz=zy^2$.

There are two simple $kG$-modules $k$ and $\epsilon$. 
Here $k$ is the trivial representation and
$\epsilon$ is a $1$-dimensional module with $z$ acting as $-1$.
Since the principal idempotent of $kG$ is $1$~\cite{B}, both $k$ and $\epsilon$ are
in the principal block.
We will show in a moment that the simple module $\epsilon$ is not in $\thick k$, hence
$\thick k \neq \stmod {B_0}$.
By Lemma~\ref{le:simple-proj}, the modules $k$ and $\epsilon$
correspond to the indecomposable projective modules sketched below:
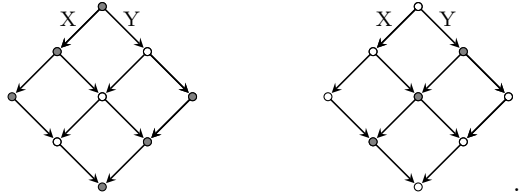
\begin{center}
\begin{tikzpicture}[scale=0.6]

\kstartt{3,1}
\knext{X}{post}{(-1,-1)}
\kXw
\eYw \kYw
\kxw \kxw
\eyw 
\knext{Y}{pre}{(-1,1)}
\kXw \eYw \kYw \kxw
\eyw \eXw \eXw
\Dend

\Estartt{10,1}
\enext{X}{post}{(-1,-1)}
\eXw
\kYw \eYw
\exw \exw
\kyw 
\enext{Y}{pre}{(-1,1)}
\eXw \kYw \eYw \exw
\kyw \kXw \kXw
\Dend
\end{tikzpicture}.
\end{center}
Here we use a solid dot for $k$ and a circle for $\epsilon$.
The arrows down-left indicate the action of $X = 1-x$, and
the arrows down-right indicate the action of $Y = y-y^2$.
Note that $Xz =zX$ and $Yz = -zY$.

With an abuse of notation, we write $\epsilon$ for both of its restrictions to $C_3 \times C_2$ and $S_3$.
Restricting to $C_3 \times C_2$, one easily sees that
$\epsilon$ is not in the principal block of $k(C_3 \times C_2)$,
hence cannot be in $\thickC {C_3 \times C_2} k$.
Since the restriction functor is triangulated, it follows that $\epsilon$ is not in $\thickC G k$,

More generally, we know that there are only $6$ indecomposable $k(C_3 \times C_2)$-modules:

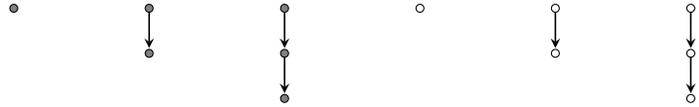
\begin{center}
\begin{tikzpicture}[scale=0.6]
\kstartt{0,1}
\Dend

\kstartt{3,1}
\knext{}{post}{(0,-1)}
\Dend

\kstartt{6,1}
\knext{}{post}{(0,-1)}
\knext{}{post}{(0,-1)}
\Dend

\Estartt{9,1}
\Dend

\Estartt{12,1}
\enext{}{post}{(0,-1)}
\Dend

\Estartt{15,1}
\enext{}{post}{(0,-1)}
\enext{}{post}{(0,-1)}
\Dend

\end{tikzpicture}.
\end{center}
Again we use a solid dot for $k$ and a circle for $\epsilon$, and
the arrows downward indicate the action of $X = 1-x$.
It is clear that the first three modules are in $\thickC {C_3 \times C_2} k$.
We know that $\epsilon$ is not in $\thickC {C_3 \times C_2} k$, and
the fifth module is isomorphic to $\Omega \epsilon$ in $\stmod{k(C_3 \times C_2)}$, hence
is not in $\thickC {C_3 \times C_2} k$ either.
The last module is projective as a $k(C_3 \times C_2)$-module, hence is in $\thickC {C_3 \times C_2} k$.
Now we can deduce the following proposition.

\begin{pro}\label{pro:S3C3}
Let $G = C_3 \times S_3$ and let $k$ be a field of characteristic $3$.
Let $M$ be a $kG$-module.
If $M$ is in $\thick k$, then the modules

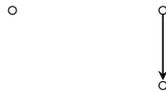
\begin{center}
\begin{tikzpicture}
\Estartt{7,1}
\Dend

\Estartt{9,1}
\enext{}{post}{(0,-1)}
\Dend
\end{tikzpicture}
\end{center}
cannot be summands of $M \down_{C_3 \times C_2}$.\qed
\end{pro}

Conversely, we can view the $k(C_3 \times C_2)$-modules as $kG$-modules with trivial $y$-action.
Again, it is easy to see that the first three modules listed above are in $\thickC G k$.
One also sees that the three-dimensional modules in the list are induced up from the subgroup $S_3$,
as $k \up^G$ and $\epsilon \up^G$.
Since $\Omega^2 k \iso \epsilon$ in $\stmod{kS_3}$, 
the last module $\epsilon \up^G$is a double suspension of the third one $k \up^G$ in $\stmod{kG}$, hence
is in $\thickC G k$ too.
But the other two modules are not in $\thickC G k$ by Proposition~\ref{pro:S3C3}.
We conjecture that the converse of the proposition is also true.
In the following example, we construct a module $M$ that satisfies the condition
in Proposition~\ref{pro:S3C3} and
show that it is in $\thickC G k$.
Indeed, this is equivalent to showing that the generating length of $M$ is finite
by Lemma~\ref{le:BoVdB}.
 
\begin{ex}\label{ex:S3C3}
We consider the cokernel $M$ of the non-zero map $f$
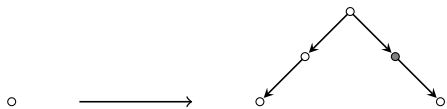
\[
\begin{tikzpicture}[scale=0.6]
\Estartt{0,1}
\Dend 
\draw[->,semithick] (1.5,1) -- (4,1);
\Estartt{5.5,1}
\exw
\exw
\kYw
\eYw
\Dend
\end{tikzpicture}
\]
that sends $\epsilon$ to the difference of the bottom elements.
By Proposition~\ref{pro:S3C3}, the domain and codomain of $f$ are not in $\thickC G k$.
But $M \down_{C_3 \times C_2}$ is in $\thickC {C_3 \times C_2} k$.
We can compute the generating length of $M$
(more precisely, an upper bound of the generating length of $M$)
with the \texttt{Length} function, and show that
\[ M \text{ is in } \thickC G k. \]
The $\texttt{Length}$ function tells us that $\gel_3(M) = 3$, and 
it follows that $\gel(M) \leq 3$.
Now we actually show that $\gel(M) = 3$.
To compute the lower bound, we consider left multiplication by the central element $1-x$
on $M$.
Restricting to $C_3 \times C_3$, we know that $1-x$ is a ghost and
$(1-x)^2$ is stably non-trivial.
Then, by Theorem~3.2 in~\cite{Gh num II}, $1-x$ is a simple ghost, hence a ghost, on $M$. 
Since the restriction functor to the Sylow $p$-subgroup is faithful, 
the generating length of $M$ is at least $3$.
\end{ex}

\end{document}